\documentclass{amsart}

\issueinfo{}% volume number
  {}%        % issue number
  {}%        % month
  {}%     % year
\pagespan{1}{}
\copyrightinfo{}%              % copyright year
  {Korean Mathematical Society}% copyright holder
\usepackage{graphicx}
\allowdisplaybreaks
\usepackage{amssymb}
\usepackage{amsmath}
\usepackage[mathscr]{euscript}
\usepackage{paralist}
\theoremstyle{plain}
\newtheorem{theorem}{Theorem}[section]

\newtheorem{lemma}[theorem]{Lemma}

\theoremstyle{definition}
\newtheorem*{definition}{Definition}
\newtheorem{example}[theorem]{Example}

\theoremstyle{remark}

\newcommand{\N}{{\mathbb N}}

\newcommand{\A}{{\mathcal A}}
\newcommand{\B}{{\mathcal B}}

\newcommand{\F}{{\mathcal F}}

\newcommand{\cpx}{{C_{p}(X)}}

\newcommand{\cpxg}{{C_{p}(X,G)}}
\newcommand{\cpyg}{{C_{p}(Y,G)}}

\newcommand{\cp}{{C_{p}}}

\begin{document}

\title[Properties of the space of group-valued continuous functions]
{Properties of the space of group-valued continuous functions}

\author[S. Mishra]{Sanjay Mishra}
\address{Sanjay Mishra \\ Department of  Mathematics, \\ Amity School of Applied Sciences, \\ Amity University Lucknow Campus, \\ Uttar Pradesh, India}
\email{drsanjaymishra1@gmail.com}

\author[P. Pandey]{Pankaj Pandey}
\address{Pankaj Pandey \\ Department of Mathematics \\  Lovely Professional University, \\ Punjab, India}
\email{pankaj.25257@lpu.co.in}

\author[S. Ravindran]{Sreeram Ravindran}
\address{Sreeram Ravindran \\ Department of Mathematics \\  Lovely Professional University, \\ Punjab, India}
\email{sreeram.r.warrier@gmail.com}

%\thanks{This work was financially supported by KRF 2003-041-C20009}


\begin{abstract}
In this paper, we find necessary and sufficient conditions for countable fan tightness and countable strong fan tightness of the space (briefly, $C_{p}(X,G)$) of all group-valued continuous functions endowed with the topology of pointwise convergence in term of Menger property and Rothberger property respectively. Furthermore, we establish a relationship between countable fan tightness, the Reznichenko property and the Hurewicz property for the space $C_{p}(X,G)$. In addition to this we prove that the Menger property is preserve during $G$-equivalence of topological spaces. Through this paper, we establish a general result regarding fan tightness of $C_{p}(X,G)$ and Hurewicz number of the space $X^{n}$ for every natural number $n$. Finally, we study the monolithicity of the space $C_{p}(X,G)$.
\end{abstract}

\maketitle

\section{Introduction and Preliminaries}\label{s:Introduction and Preliminaries}
In 1992, Arkhangelski \cite{Arhangelskii1992} introduced a theory called $\cp$-theory for topological function spaces. Subsequently, many mathematicians made considerable efforts to enhance $\cp$-theory, giving it the sophistication and elegance it currently has. Arkhangelsky's PhD student, Tkachuk, has written a wide range of books \cite{Tkachuk2010, Tkachuk2014, Tkachuk2015, Tkachuk2016}, which serve as a vast compilation of many findings related to $\cp$-theory. In these book series, Tkachuk has produced a large collection of open problems that serve as an attractive stimulus for further research, not only to advance $\cp$-theory but also to meet the challenges of other branches of mathematics. Building on this line of inquiry, a recent monograph by McCoy  \cite{McCoy2018} highlights the broad aspects of the $\cp$-theory, and provides information about its more general implications. His study includes a comprehensive exploration of many properties related to the space of continuous functions from one topological space to another, culminating in the examination of uniform, fine, and graph topology. More recently, in 2023, Mishra and Bhaumik \cite{Mishra2023} and Aaliya and Mishra \cite{Aaliya2023, Aaliya2023a, Aaliya2024} studied properties of topological function spaces under Cauchy convergence topology and regular topology, respectively. In 2024, Bishnoi and Mishra \cite{Bishnoi2024} studied cardinal properties of the space of quasi-continuous functions under topology of uniform convergence on compact subsets. These studies also give an idea for further studies on topological function space by applying the new structure. However, such a study was initiated in 2010 by Shakhmatov and Spevak \cite{Shakhmatov2010}, considering a new structure on topological function spaces as topological groups. In this paper \cite{Shakhmatov2010} authors defined point-wise convergence topology on the space of group-valued continuous functions (denoted by $\cpxg$) and further studied on good numbers of properties of such space along with preservation of properties during some equivalencies like $G$-equivalence and $\mathbb{T}$-equivalence. In 2011, Kocinac \cite{Kocinac2011} extended the closure-type properties of the function spaces of continuous real-valued functions on a Tykonoff space $X$ to the group-valued function spaces $\cpxg$.

In this paper we are going to study about countable fan tightness, countable strong fan tightness and fan tightness of the space $\cpxg$ in terms of Menger property, Rothberger property, Reznichenko property and Hurewicz property. In 2020, Sakai \cite{Sakai2020} studied about the preservation of the Menger property during $l$-equivalence between spaces of real-valued continuous functions. Furthermore we examined the preservation of the Menger property for the space of group-valued continuous functions $\cpxg$ in the light of $G$-equivalence which is a more general case of $l$-equivalence. We also studied the monolithicity of the space $\cpxg$.

We are following most of the notation and terminology from \cite{Engelking1989} and \cite{Shakhmatov2010}, unless we state otherwise. The study assumes that all topological spaces under consideration are Tychonoff, meaning they are both completely regular and satisfy the $T_{1}$ separation axiom. Additionally, it is assumed that these spaces are non-empty. Furthermore, all topological groups considered are assumed to be Hausdorff. The $\N$ is the set of all natural numbers, $\omega$ is the least non-zero limit ordinal. For a space $X$, we denote by $\cpxg$ the space of all group-valued continuous functions with the topology of point-wise convergence, where $G$ is topological group.

The selection principle refers to a guiding principle that affirms the feasibility of obtaining mathematically significant objects by selecting elements from predetermined sequences of sets. Selection theory primarily involves the characterization of covering properties, measure-theoretic properties, category-theoretic properties, and local properties within topological spaces, with a particular emphasis on function spaces. These theories provide a framework for understanding and analyzing the behavior of sets and functions with respect to these properties. Frequently, employing selection theory to characterize mathematical properties presents a challenging endeavor, which often yields fresh insights into the distinctive nature of the property being studied. In this paper we are going to apply the results of selection principles to discuss the tightness of the $\cpxg$ space. In order to better understand our results we are going to briefly recall some important concepts such as selection principles, Menger, Hurewicz, Rothberger and Reznichenko properties, but for detail readers can refer to the following paper \cite{Scheepers1996,Fremlin1988}.

\begin{definition}[Selection principles]\label{d:Selection principles}[\cite{Kocinac2003}, Section 1.1]
Let $S$ be an infinite set and let $\A$ and $\B$ are families of the subsets of $S$. The two selection hypotheses $\textsf{S}_{1}(\A, \B)$ and $\textsf{S}_{fin}(\A, \B)$ are defined as follows:
\begin{enumerate}
\item Notation $\textsf{S}_{1}(\A, \B)$ denote the statement:Corresponding to every sequence $(A_{n})_{n \in \N} \in \A$ we can find a sequence $(b_{n})_{n \in \N}\in  A_{n}$ and $\{b_{n} \colon n \in \N\} \in \B$ for every $n\in \N$.

\item Notation $\textsf{S}_{fin}(\A, \B)$ denote the statement: Corresponding to every sequence $(A_{n})_{n \in \N} \in \A$ we can find a sequence of finite sets $B_{n} \subset A_{n}$, and $\bigcup_{n < \infty}B_{n} \in \B$ for every $n\in \N$.
\end{enumerate}
\end{definition}

\begin{definition}[Menger space]\label{d:Menger space}[\cite{Sakai2020}, Definition 1.2]
A space $X$ is Menger if for every sequence  of open covers of $(U_{n})_{n \in \omega}$ of $X$ we can find a sequence of finite sets $V_{n} \subset U_{n}$ and $\bigcup\{V_{n} \colon n \in \omega\}$ is a cover of $X$.
\end{definition}

\begin{definition}[Countable fan tightness space]\label{d:Countable fan tightness space}[{\cite{Kocinac2011}, Section 1.1}]
A space $X$ has countable fan tightness if $\textsf{S}_{fin}(\Omega_{x}, \Omega_{x})$ holds for each $x \in X$, where $\Omega_{x}= \{A \subset X \backslash \{x\} \colon x \in \bar{A}\}$.
\end{definition}

\begin{example}\label{eg:Countable fan tightness}[{\cite{Kocinac2011}, Section 1.1}]
The space $C_{p}(X, G)$ has countable fan tightness, where the $X = [0, 1]$ is set with usual topology and $G = S^{1}$ is unit circle in the complex plane with the usual topology.
\end{example}

\begin{definition}[$\omega$-cover]\label{d:omega cover}[{\cite{Kocinac2011}, Section 1.1}]
An open cover $\mathscr{U}$ of a space $X$ is said to be an $\omega$-cover if for every finite set $F$ in $X$ we can find a $U \in \mathscr{U}$ such that $F\subset U$ and $X \notin \mathscr{U}$ .
\end{definition}

\begin{definition}[Groupable $\omega$ cover]\label{d:Groupable cover}[{\cite{Kocinac2011}, Section 1.1}]
A $\omega$-cover $\mathscr{U}$ of a space $X$ is groupable if there is a partition $\{U_{n} \colon n \in \N\}$ of $\mathscr{U}$ into pairwise disjoint finite sets such that for each compact subset $K$ of $X$, for all but finitely many $n$, there is a $U \in U_n$ such that $K \subset U$. An element $A$ of $\Omega_{x}$ is groupable if there is a partition $\{A_{n} \colon n \in \N\}$ of $A$ into pairwise disjoint finite sets such that each neighborhood of $x$ has nonempty intersection with all but finitely many of the $A_{n}$. The set $\Omega_{gp}$ denotes the collection of all groupable $\omega$-covers of a space.
\end{definition}

\begin{definition}[Hurewicz covering property]\label{d:Hurewicz covering property}[{\cite{Kocinac2011}, Section 1.1}]
A space $X$ has the Hurewicz covering property if for each sequence $(U_{n})_{n \in \N}$ of open covers of $X$ there is a sequence $(V_{n})_{n \in \N}$ with $V_{n}$ is a finite subset of $U_{n}$ for each $n \in \N$, and each $x \in X$ belongs to $\bigcup V_{n}$ for all but finitely many $n$.
\end{definition}

\begin{definition}[Rothberger space]\label{d:Rothberger space}[\cite{Sakai2020}]
A space $X$ is said to be Rothberger if for any sequence $(\mathscr{U}_{n})_{n \in \omega}$ of open covers of $X$, there exist some $U_{n} \in \mathscr{U}_{n \in \omega}$ such that $\{U_{n} \colon n \in \omega\}$ is a cover of $X$.
\end{definition}

\begin{definition}[Reznichenko property]\label{d:Reznichenko property}[\cite{Kocinac2003}, Section 4]
A space $X$ has the Reznichenko property at $x \in X$ if for each $A \in \Omega_{x}$ there is a countably infinite pairwise disjoint collection $\F$ of finite subsets of $A$ such that for each neighborhood $U$ of $x$, for all but finitely many $F \in \F$, $U \cap F$ is nonempty. When $X$ has the Reznichenko property at each of its points, then we say that $X$ has the Reznichenko property.
\end{definition}
The following implication is true for any arbitrary space.
\begin{gather*}
\sigma-\text{compact space} \Rightarrow \text{Hurewicz space} \Rightarrow \text{Menger space} \Rightarrow \text{Lindelöf space} \\  \text{Rothberger space} \Rightarrow \text{Menger space}
\end{gather*}
\begin{definition}[Network weight of a space, \cite{Arhangelskii1992}]
For any space $X$, a family $\mathcal{N}$ of subsets of $X$ is a network in the space $X$ if for any $x \in X$ and for any open subset $U$ of $X$ with $x \in U$, there exists some $J \in \mathcal{N}$ such that $x \in J \subset U$.

In other words, any non-empty open subset of $X$ is the union of elements of the network $\mathcal{N}$. The network weight of $X$, denoted by $nw(X)$, is the minimum cardinality of a network in the space $X$, i.e., $nw(X)$ is the least cardinal number $\kappa$ such that if $\mathcal{N}$ is a network for the space $X$, then $\kappa \le \lvert \mathcal{N} \lvert$.
\end{definition}

\begin{definition}[Weight of a space, \cite{Arhangelskii1992}]
For any space $X$, the weight of $X$, denoted by $w(X)$, is the minimum cardinality of a base for the space $X$, i.e., $w(X)$ is the least cardinal number $\kappa$ such that if $\mathcal{B}$ is a base for the space $X$, then $\kappa \le \lvert \mathcal{B} \lvert$.
\end{definition}

Note that if $w(X)=\omega$, then $X$ is a space with a countable base (it is a separable metric space). Also if $nw(X)=\omega$, then $X$ is a space with a countable network. 
The cardinal invariant $iw(X)$ is called the $i$-weight of $X$, which is a minimal weight of all spaces onto which $X$ can be condensed.
\begin{definition}[$\tau$-monolithic space, \cite{Arhangelskii1992}]
Let $\tau$ be an infinite cardinal number. A space $X$ is said to be $\tau$-monolithic if for each subspace $Y$ of $X$ with $\lvert Y \lvert \le \tau, nw(\overline{Y}) \le \tau$.
\end{definition}

\begin{definition}[Monolithic space]
A space $X$ is said to be monolithic if it is $\tau$-monolithic for each infinite cardinal $\tau$.
\end{definition}

\begin{example}
The space $\mathbb{N}$ with cofinite topology is $\tau$-monolithic for each infinite cardinal $\tau$. Hence it is a monolithic space.
\end{example}

\begin{example}
$C_{p}([0,1])$ is a monoloithic space.
\end{example}

\begin{theorem}[\cite{Arhangelskii1992}]\label{t:Equivalent condition for monolithic space}
The following statements are equivalent:
\begin{enumerate}
\item $X$ is $\tau$-monolithic for each infinite cardinal number $\tau$.
\item For each subspace $Y$ of $X$, $d(Y)=nw(Y)$.
\end{enumerate}
\end{theorem}

Notation:
\begin{definition}[Stable space, \cite{Arhangelskii1992}]
A space $X$ is called $\tau$-stable if for every continuous image $Y$ of $X$,  $iw(Y)<\tau$. A space $X$ is called stable if it is $\tau$-stable for every infinite cardinal $\tau$.
\end{definition}

\begin{example}
Every compact space is stable.
\end{example}

\begin{definition}[\cite{Shakhmatov2010}, Definition 2.2]
For a topological group $G$, we say that a space $X$ is

\begin{enumerate}
\item $G$-regular, provided that corresponding to every closed subset $F$ of $X$ and for each point $x\in X\backslash F$, we can find $f\in\cpxg$ and a non identity element $g\in G$ such that $f(x) =g$ and $f(x)\in \{e\}$ for all $x\in F$.

\item $G^{\ast}$-regular, provided that there exists a non identity element $g\in G$ such that corresponding to every closed subset $F$ of $X$ and for every $x \in X \backslash F$, we can find $f\in\cpxg$ such that $f(x) =g$ and $f(x)\in \{e\}$ for all $x\in F$.

\item $G^{\ast\ast}$-regular, provided that for every closed subset $F$ of $X$, for each $x\in X\backslash F$ and for every element $g$ in $G$, we can find $f\in\cpxg$ such that $f(x)=g$ and $f(x)\in \{e\}$ for every $x\in F$.
\end{enumerate}
\end{definition}

\section{Tightness property of the space $\cpxg$}\label{s:Tightness property of the space cpxg}
In 1986,  \cite{arkhangelskii1986} Arkhangelskii find the necessary and sufficient condition for countable fan tightness of the space $\cpx$ in term of Menger property of the space $X$. In continuation of this study we are going to generalize this  result for the space $\cpxg$ under the certain condition. In this section we are going to find necessary and sufficient conditions for the space $\cpxg$ to be countable fan tightness (countable strong fan tightness) in term of Menger (Rothberger) property of the base space $X$. We also include the necessary and sufficient condition for the space to have  Reznichenko property in terms of Hurewicz property. Following results motivates us to find such conditions.

\begin{theorem}[\cite{Kocinac2011}, Corollary 2.4]\label{t:NS condition for count fan tightness of cpxg}
Let $G$ be a metric group and $X$ be a $G^{\ast}$-regular space then $\cpxg$ has countable fan tightness if and only if $X$ satisfies $S_{fin}(\Omega,\Omega)$.
\end{theorem}

\begin{theorem}[\cite{Kocinac2011}, Theorem 2.5]\label{t:NS condition for count strong fan tightness of cpxg}
Let $G$ be a metric group and $X$ be a $G^{\ast}$-regular space then $\cpxg$ has  countable strong fan tightness if and only if $X$ satisfies $S_{1}(\Omega,\Omega)$
\end{theorem}

\begin{definition}[\cite{Kocinac2011}]
A space $X$ is said to be $\omega$-Lindelof if each $\omega$-cover of $X$ contains a countable
$\omega$-subcover. Equivalently, all finite powers of $X$ are Lindelof.
\end{definition}

\begin{theorem}[\cite{Kocinac2003Scheepers}]\label{t:NS condition Hurewicz property}
Let $X$ be a $\omega$-Lindelof space, then  $X$ satisfies $S_{fin}(\Omega,\Omega^{gp})$ if and only if $X^{n}$ has Hurewicz property for every natural number $n$.
\end{theorem}

\begin{theorem}\label{t:NS condition for countable fan tightness of cpxg in term of Menger property}
Let $X$ is a metrizable $\omega$-lindeloff space and $G$ be a metric group. Let $X$ is $G^{\ast}$-regular then $\cpxg$ has countable fan tightness if and only if $X^{n}$ has Menger property for all $n \in \N$.
\end{theorem}

\begin{proof}
Let $\cpxg$ has countable fan tightness, then from Theorem \eqref{t:NS condition for count fan tightness of cpxg}, $X$ satisfies $S_{fin}(\Omega,\Omega)$. So from Theorem \eqref{t:NS condition Hurewicz property}, $X^n$ has Hurewicz property for every $n\in\mathbb{N}$. Since Hurewicz property implies Menger property for every topological space \cite[Defintion 1.2]{Sakai2020}, $X^{n}$ has Menger property for every natural number $n$.

Conversely assume that $X^{n}$ has Menger property for every natural number $n$. Then from \cite[Theorem 14]{Kocinac2003Scheepers}, $X$ satisfies $S_{fin}(\Omega,\Omega)$. Therefore from Theorem \eqref{t:NS condition for count fan tightness of cpxg}, $\cpxg$ has countable fan tightness.
\end{proof}

\begin{definition}[\cite{Gerlits1982}]
Let $\phi = <\mathscr{G}_{n},n\in \omega>$ be a sequence of open covers of the space $X$. A set $A \subseteq X$ is said to be $\phi$-small if for any $n\in \omega$ there are $k \in \omega$ and sets $G_{i} \in \mathscr{G}_{n+i}$ such that $A \subset \cap \{G_{i}: i < k\}$.
\end{definition}

Let us see the  property that was introduced by Gerlits and Nagy in \cite[Theorem 5]{Gerlits1982}:
\begin{multline}\label{property:1}
\text{If} \, \phi=<\mathscr{G}_{n}, n\in \omega> \, \text{is a sequence of open covers of} \,  X, \text{then} \, X \, \text{is the union of} \\ \text{countably many}\, \phi-\text{small sets}.\tag{*}
\end{multline}

\begin{theorem}[\cite{Kocinac2003Scheepers}, Theorem 19]\label{t:Gerlits Nagy condition}
For an $\omega$-Lindelof space $X$, $X$ has property $S_1(\Omega, \Omega^{gp})$ if and only if each finite power of $X$ has property \eqref{property:1}.
\end{theorem}

Using the above theorem we can generalize the result of \cite[Theorem 1]{Sakai1988}, in $C_{p}$-theory that $\cpx$ has countable strong fan tightness if and only if every finite power of $X$ is Rothberger (equivalent to property $C{''}$ in \cite{Sakai1988}).

\begin{theorem}
Suppose $X$ is an $\omega$-Lindelof space and $G$ be a metric group. Assume $X$ is $G^{\ast}$-regular then $\cpxg$ has countable strong fan tightness if and only if $X^{n}$ has Rothberger property for every natural number $n$.
\end{theorem}

\begin{proof}
Let that $\cpxg$ has countable strong fan tightness, then from Theorem \eqref{t:NS condition for count strong fan tightness of cpxg}, $X$ satisfies $S_{1}(\Omega, \Omega)$. Then from Theorem \eqref{t:Gerlits Nagy condition}, since $X$ is a $\omega$-Lindelof space $X$ satisfies property \eqref{property:1}. In a $\omega$-lindelof space property \eqref{property:1} implies Rothberger property. Therefore $X^{n}$ satisfies Rothberger property for every natural number $n$.

Conversely assume that $X^n$ satisfies Rothberger property for every natural number $n$, In a $\omega$-Lindelof space Rothberger property implies property \eqref{property:1}, so from Theorem \eqref{t:Gerlits Nagy condition}, $X$ satisfies $S_{1}(\Omega, \Omega)$. Then, from Theorem \eqref{t:NS condition for count strong fan tightness of cpxg}, $\cpxg$ has countable strong fan tightness.
\end{proof}

Now we will generalize another result on $C_{p}$-theory that $\cpx$ has countable fan tightness and the Reznichenko property if and only if all finite powers of $X$ have the Hurewicz property (\cite{Kocinac2003Scheepers}, Theorem 21). To prove the following result we need the following Lemma \cite[Lemma 2.1]{Kocinac2011}.

\begin{lemma}
\label{l:property of G regular space}
 If $G$ is a topological group and $X$ is a $G^{\ast}$-regular space, then there is $g \in G \backslash \{e\}$ such that for each open set $U \subset X$ and each non-empty finite set $F \subset U$ there is $f_{F,U} \in C_{p}(X, G)$ satisfying $f_{F,U} (F) \subseteq \{e\}$ and $f_{F,U} (X \backslash U) \subseteq \{g^{-1}\}$.
\end{lemma}

\begin{theorem}\label{l:NS condition for reznichenko property in terms of hurewicz property}
Let $X$ be an $\omega$-Lindelof space and $G$ a metric group. Let $X$ be $G^{\ast}$-regular then $\cpxg$ has countable fan tightness and Reznichenko property if and only if each finite power of $X$ has Hurewicz property.
\end{theorem}

\begin{proof}
Suppose that $\cpxg$ has countable fan tightness and Reznichenko property.
Let $U_{n}$ be a sequence of $\omega$-covers of $X$. Let $F$ be a finite subset of $X$ then since $U_{n}$ is an $\omega$-cover of $X$ we can find $U_{n,F}\in U_{n}$ such that $F\subset U_{n,F}$. Then by lemma \eqref{l:property of G regular space} we can find $g\in G$ and $f_{F,U_{n,F}}\in\cpxg$
such that $f_{F,U_{n,F}}(F)=\{e\}$, and $f_{F,U_{n,F}}(X\backslash U)\subset \{g^{-1}\}$. Take $A_{n}=\{f_{F,U_{n,F}}: F\subset X$ is finite and $U\in U_{n}\}\}$. Then $f_{e}\in \overline{A_{n}}$ for every $n\in \mathbb{N}$.

Since $\cpxg$ have Reznichenko property, for each $n \in \mathbb{N}$ we can find a sequence of finite subsets $B_n$ of $A_{n}$ and for each neighborhood $U$ of $f_{e}$, $U\cap B_{n}\neq \phi$, for finitely many $n\in\mathbb{N}$. Let $V_{n}$ be the sequence of sets of the form $U_{n,F}$ such that $f_{F,U_{n,F}}\in B_{n}$. Then $V_{n}\subset U_{n}$ for every $n\in \mathbb{N}$. Let $D$ be a finite subset of $X$. Now we will show that $E\subset V$ for some $V\in V_{n}$. Consider the neighborhood $N(D,O_{k})$ of $f_{e}$ such that $N(D,O_{k})\cap B_{n}\neq \phi$ for $n\geq n_{0}$ for some $n_{0}\in\mathbb{N}$. Let $f_{F,U_{n,F}}\in N(D,O_{k})$, that is for $x\in D,f_{F,U_{n,F}}(x)\in O_{k} $. Therefore $D\subset U_{n,F}\in V_{n}$. This proves Hurewicz property of $X^{n}$ for each $n\in\mathbb{N}$.

Conversely assume that $X^{n}$ has Hurewicz property for every $n\in\mathbb{N}$. Let $\O_{n} : n\in\mathbb{N}$ be the countable local base at $e \in G$. Let $A_{n}$ be a sequence of subsets of $\cpxg$ such that $f_{e}\in\cap_{n\in\mathbb{N}}\overline{A_{n}}$. Let $D$ be a finite subset of $X$. Then the neighborhood $N(D,O_{1})$ of $f_{e}$ has non-empty intersection with $A_{1}$. Take $f_{F,1}\in A_{1}$. Since $f_{F,1}$ is continuous, choose for each $x\in F$ an open set $V_{x}$ such that $f_{F,1}(V_{x})\subset O_{1}$, and set $V_{F,1}=\cup_{x\in F}V_{x}$. Then the collection $\mathcal{V}_{1}=\{V_{F,1}:F\subset X$ finite $\}$. Clearly $\mathcal{V}_{1}$ is a $\omega$ cover of $X$. Similar way we can construct $\mathcal{V}_{n}$ for $n\geq 2$. Since $X$ satisfies Hurewicz property and also by using theorem \eqref{t:NS condition Hurewicz property} there exists finite sets $\mathcal{W}_{n}\subset \mathcal{V}_{n},n\in \mathbb{N}$ such that every finite subset of $X$ is a member of $\mathcal{W}_{n}$ for all but finitely many $n$. Now without loss of generality, we can assume that $\mathcal{W}_{n}'s$ are pairwise disjoint. That is, $W=\cup_{n\in\mathbb{N}}\mathcal{W}_{n}$ is a groupable $\omega$ cover of $X$. Let $\mathcal{W}_{n}=\{ V_{F_{1}^{(n)}}, V_{F_{2}^{(n)}}, \hdots V_{F_{k}^{(n)}}\}$, $n\in \mathbb{N}$. For $n\in\mathbb{N}$, let $B_{n}\subset A_{n}$ be the set of all functions in $A_{n}$ such that $V_{F_{i}^{(n)}}\in \mathcal{W}_{n}, i\leq k_{n}$. Then $f_{F,i}(V_{x})\subset O_{i}$. This implies $\cpxg$ satisfies Reznichenko property. Countable tightness of $\cpxg$ will implies from theorem \eqref{t:NS condition for count fan tightness of cpxg} and \eqref{t:NS condition Hurewicz property}.
\end{proof}

Let us see the fan tightness and Hurewicz number defined in \cite{Lin2006} by  Lin.
\begin{definition}[\cite{Lin2006}]
The fan tightness of a space $X$ is defined by \\$vet(X) = \sup\{vet(X, x) \colon x \in X\}$, where $vet(X, x) = \omega + \min\{\lambda \colon$for each family $\{A_{\gamma}\}_{\gamma < \lambda}$ of subsets of $X$ with $x \in \bigcap_{\gamma < \lambda}\overline{A_{\gamma}}$ there is a subset $B_{\gamma}\subset A_{\gamma}$ with $\lvert B_{\gamma} \rvert<\lambda$ for each $\gamma < \lambda$ such that $x \in \overline{\bigcup_{\gamma<\lambda}{B_{\gamma}}\}}$.
\end{definition}

Note that, A space $X$ has countable fan tightness if and only if $vet(X) =\omega$.
\begin{definition}[\cite{Lin2006}]
Let $\alpha$ be a network of compact subsets of a space $X$, which is closed under finite unions and closed subsets. An $\alpha$-cover of a space $X$ is a family of subsets of $X$ such that every member of $\alpha$ is contained in some member of this family. An $\alpha$-cover is called a $k$-cover if $\alpha$ is the set of all compact subsets of $X$. An $\alpha$-cover is called a $\omega$-cover if $\alpha$ is the set of all finite subsets of $X$. The $\alpha$-Hurewicz number of $X$ is defined by $\alpha H(X)= \omega + \min \{\lambda \colon$ for each family $\{ \mathcal{U}_{\gamma}\} _{\gamma < \lambda}$ of open $\alpha$-covers of $X$ there is a subset $\mathcal{B}_{\gamma}\subset \mathcal{U}_{\gamma}$ with $\lvert \mathcal{B}_{\gamma} \rvert < \lambda $
for each $\gamma < \lambda$ such that $\bigcup_{\gamma < \lambda}\mathcal{B}_{\gamma}$ is an $\alpha$ cover of $X\}$.

The $\alpha$-Hurewicz number of $X$ is called the Hurewicz number of $X$ and written $H(X)$ if $\alpha$ consists of the singleton of $X$. A space $X$ is Hurewicz space if and only if $H(X) = \omega$.
\end{definition}

\begin{theorem}
Let $G$ be a metric group and $X$ be a $G^{\ast}$-regular space. Then $vet(\cpxg)= \sup \{H(X^{n})\}$, for every natural number $n$.
\end{theorem}

\begin{proof}
Assume that $vet(\cpxg)=\lambda$. We need to show that $\sup\{H(X^{n})\}=\lambda$, for every natural number $n$. To see this, let $\{\mathcal{U}_{\gamma}\}_{\gamma < \lambda}$ be a family of open covers of the space $X^{n}$. Now, we define a property $P_{n,\gamma}$ for the family $\mathcal{V}$ of subsets of $X$ that for every $\gamma < \lambda$ and for each $\{V_{i}\}_{i\leq n}\subset \mathcal{V}$ there is $U\in \mathcal{U}_{\gamma}$ such that $\Pi_{i=1}^{n}V_{i}\subset U$. Let $I_{n,\gamma}$ denote the family of finite open sets having the property $P_{n,\gamma}$. Now for each $\mathcal{V}\in I_{n,\gamma}$ and for $V\in \mathcal{V}$ there exists $U\in \mathcal{U}_{\gamma}$ such that $V\subset U$. By applying Lemma \eqref{l:property of G regular space}, there exists $g\in G$ and $f_{V,\gamma}\in\cpxg$ such that $f_{V,\gamma}(V)=\{e\}$ and $f_{V,\gamma}(X \backslash U)\subset\{g^{-1}\}$. Let us define $F_{\mathcal{V}}= \{ f_{V,\gamma}\in\cpxg :V\in \mathcal{V}$ and $\mathcal{V}\in I_{n,\gamma} \}$. Then, we claim that the set $A_{\gamma} = \bigcup_{\gamma < \lambda}F_{\mathcal{V}}$ is dense in $\cpxg$.

Let $W(E,O)$, (where $E$ a finite subset of $X$, and $O$ is a neighbourhood of $e\in G$), be a neighbourhood of $f_{e}$ in $\cpxg$. Since $E$ is finite, there exists $\mathcal{W}\in I_{n,\gamma}$ such that for any $(x_{1}, x_{2}, \ldots, x_{n})\in E^{n}$, there are $U \in \mathcal{U}_{\gamma}$ and a finite subset $ \{W_{i}\}_{i\leq n}\subset \mathcal{W}$ such that $(x_{1}, x_{2}, \ldots,x_{n}) \in \Pi_{i\leq n}W_{i}\subset U$. Then, $E\subset \bigcup \mathcal{W}$. Now for each $x\in E$, we define $V_{x}= \bigcap\{ W\in \mathcal{W}\colon x\in W\}$ and put $\mathcal{V}=\{ V_{x}: x\in E\}$. It is clear that the family $\mathcal{V}$ has the property $P_{n,\gamma}$ and $E\subset \bigcup\mathcal{V}$. As we already noticed that  for each  $(x_{1}, x_{2}, \ldots, x_{n}) \in E^{n}$, there are $U \in \mathcal{U}_{\gamma}$ and a finite subset $ \{W_{i}\}_{i\leq n}\subset \mathcal{W}$ such that $(x_{1}, x_{2}, \\dots, x_{n})\in \Pi_{i\leq n}W_{i}\subset U$. Since $V_{x_{i}}\subset W_{i}$, we have $\Pi_{i\leq n} V_{x_{i}}\subset U$. Choose $h\in\cpxg$ such that $f(E)=h(E)$ and $h(X\backslash U) = g^{-1}$. Then, $h\in F_{\mathcal{V}}\subset A_{\gamma}$, so $W(E,O)\bigcap A_{\gamma}\neq \phi$. Therefore, $A$ is dense in $\cpxg$.

Now fix a $g\in G$ and define $f_{g}\in C(X,G)$ such that $f_{g}(X)=g$. Then, $f_{g}\in\bigcap_{\gamma<\lambda}\overline{A_{\gamma}}$. By applying the Definition of $vet(X)$, there is a subset $B_{\gamma}\subset A_{\gamma}$ with $\lvert B_{\gamma} \rvert<\lambda$ for each $\gamma < \lambda$ such that $f_{g}\in \overline{\bigcup_{\gamma < \lambda}B_{\gamma}}$. We can find a subset $J_{n,\gamma}$ of $I_{n,\gamma}$ with $\lvert J_{n,\gamma} \rvert < \lambda$ such that $B_{\gamma}\subset \bigcup\{F_{\mathcal{V}}\colon \mathcal{V}\in J_{n,\gamma}\}$. Since $J_{n,\gamma}$ satisfies property $P_{n,\gamma}$, we have for $\mathcal{V}\in J_{n,\gamma}$ and for each $\psi = (V_{1}, V_{2}, \ldots, V_{n})\in \mathcal{V}^{n}$, take $M_{\psi}\in \mathcal{U}_{\gamma}$ such that $\Pi_{i\leq n}V_{i}\subset M_{\psi}$. Put $\mathcal{M}_{\gamma}=\{ M_{\psi}\colon \psi\in \mathcal{V}^{n},\mathcal{V}\in J_{n,\gamma}\}$. Obviously $\lvert \mathcal{M}_{\gamma} \rvert < \lambda$ and $\mathcal{M}_{\gamma}\subset \mathcal{U}_{\gamma}$. Now, we will show that $\bigcup_{\gamma < \lambda}\mathcal{M}_{\gamma}$ covers $X$.

Let $(x_{1}, x_{2}, \ldots, x_{n})\in X^{n}$ and $N$ be an open neighborhood of $f_{g}\in\cpxg$. Since $f_{g}\in \overline{\bigcup_{\gamma < \lambda}B_{\gamma}}$, there is $\gamma < \lambda$ such that $N\cap B_{\gamma}\neq \phi$, then $N\cap F_{\mathcal{V}}\neq \phi$ for some $\mathcal{V}\in I_{n,\gamma}$. Let $z\in N\cap  F_{\mathcal{V}}$, then $z(X\backslash U)=g^{-1}$ and $z(x_{i})\in N$ for every $i=1,2,\hdots,n$. Choose $V_{i}\in \mathcal{V}$ such that $x_{i}\in V_{i}$ for each $i\leq n$,  then there exists $M_{\psi}\in \mathcal{M}_{\gamma}$ such that $(x_{1}, x_{2},\hdots,x_{n})\in \Pi_{i\leq n}V_{i}\subset M_{\psi}$. So $(x_{1}, x_{2}, \ldots, x_{n})\in\bigcup(\bigcup_{\gamma<\lambda}\mathcal{M}_{\gamma})$. Hence, $H(X^{n})\leq vet(C_{p}(X,G))$.

Conversely, assume that $\sup\{H(X^{n})\}=\lambda$. Fix a decreasing local base $\{O_{n}\colon n\in \mathbb{N}\}$ for $e\in G$. Let $\{A_{\gamma}\}_{\gamma<\lambda}$ be a family of subsets of $\cpxg$ such that $f_{e}\in\bigcap_{\gamma<\lambda} \overline{A_{\gamma}}$. For each finite set $F$ of $X$ and $\gamma<\lambda$ the neighborhood $W(F,O_{1})$ of $f_{e}$ has non-empty intersection with $A_{\gamma}$. Choose $h_{x,\gamma}\in W(F,O_{1})\bigcap A_{\gamma}$. Since $h_{x,\gamma}$ is continuous, for each $x_{i}\in F$, there exists an open set $V_{x_{i}}$ such that $h_{F,\gamma}(V_{x_{i}})\subset O_{1}$. Let $U_{x,\gamma}=\Pi_{i=1}^{n}V_{x_{i}}$ be a neighborhood of $x=(x_{1}, x_{2}, \ldots, x_{n}) \in X^{n}$, then $\mathcal{U}_{n,\gamma}=\{U_{x,\gamma}\colon x\in X^{n}\}$ covers $X^{n}$. Note that for each $(y_{1}, y_{2}, \ldots, y_{n})\in U_{x,\gamma}$, $h_{x,\gamma}(y_{i})\in O_{i}$.
\begin{description}
\item[Case-I] Suppose $\lambda >\omega$. Since $H(X^{n})\leq\lambda$, we can find a family of subsets $\{ S_{n,\gamma}\}$ in $X^{n}$ with $\lvert S_{n,\gamma}\rvert < \lambda$ for each $\gamma< \lambda$ such that $\bigcup_{\gamma<\lambda}S_{n,\gamma}$ covers $X^{n}$. Note that $S_{n,\gamma}=\{U_{x,\gamma}: x\in S_{n,\gamma}\}$. Let us define for each $\gamma<\lambda,the sets B_{n,\gamma}=\{h_{x,\gamma}:x\in S_{n,\gamma}\}$ and $B_{\gamma}= \bigcup_{n\in\mathbb{N}} B_{n,\gamma}$, then $B_{\gamma}\subset A_{\gamma}$ with $\lvert B_{\gamma}\rvert<\lambda$, and $f\in\overline{\bigcup_{\gamma<\lambda}B_{\gamma}}$. Let $W(E,O)$, be a basic neighborhood of $f_{e} \in \cpxg$. Then, $\gamma<\lambda$ such that $(y_{1}, y_{2}, \ldots, y_{n})\in \bigcup S_{n, \gamma}$ and $x\in  S_{n, \gamma}$ such that $(y_{1}, y_{2}, \ldots, y_{n})\in U_{x,\gamma}$. Therefore, $h_{x,\gamma}\in B_{n,\gamma}$ and $h_{x,\gamma}\in O_{n}$ for each $i\leq n$. Hence $h_{x,\gamma}\in W(E,O)$,  (where $E=\{y_{1}, y_{2}, \ldots, y_{n}\}$) i.e. $h_{x,\gamma}\in W(E,O)\cap B_{\gamma}$. Therefore  $f_{e}\in\overline{\bigcup_{\gamma<\lambda}B_{\gamma}}$.
\item[Case-II] Suppose $\lambda =\omega$. Replace $\gamma$ with a natural number $k\geq n$. Choose $B_{k}=\bigcup_{n\leq k}B_{n,k}$ and follow Case-I. The proof will be immediate.
\end{description}
So $vet(C_{p}(X,G))\leq \sup\{H(X^{n})$ and therefore $vet(C_{p}(X,G))=\sup\{H(X^{n})\}$
\end{proof}

\section{Preservation of Menger property on $G$-equivalence}\label{s:Preservation of Menger property on G-equivalence}
\begin{definition}[\cite{Shakhmatov2010}] Two topological spaces $X$ and $Y$ are said to be $G$-equivalent if $\cpxg \cong \cpyg$.
\end{definition}

If $G=\mathbb{T}$, the circle group $\mathbb{R}\backslash \mathbb{Z}$, then $G$-equivalence can be viewed as $\mathbb{T}$-equivalence. Since $\mathbb{T}$ can be viewed as a counterpart to $\mathbb{R}$ in the theory of topological vector spaces, this line of research has significant importance in the theory of group valued continuous functions.

\begin{definition}[\cite{Shakhmatov2010}]
Two topological spaces $X$ and $Y$ are said to be $\mathbb{T}$-equivalent if $C_{p}(X,\mathbb{T}) \cong C_{p}(Y,\mathbb{T})$.
\end{definition}

\begin{theorem}[\cite{Shakhmatov2010}, Theorem 10.7]\label{t:Preservation of T equivalence}
$\mathbb{T}$ equivalence preserves the following properties:
\begin{enumerate}
\item pseudocompactness,
\item the cardinal invariant $l^{*}$ (defined in \cite{arkhangelskii1986}),
\item property of being a Lindelöf $\sum$-space,
\item $\sigma$-compactness,
\item compactness,
\item the property of being compact metrizable,
\item the (finite) number of connected components,
\item connectedness,
\item total disconnectedness.
\end{enumerate}
\end{theorem}

\begin{lemma}[$\mathbb{T}$ equivalence preserves Menger property]\label{l:Menger property on T equivalence}
Suppose $X$ and $Y$ are $\mathbb{T}$-equivalent. If $X$ is a Cech complete Menger space. then $Y$ is also a Menger space.
\end{lemma}

\begin{proof}
Let $X$ be a Cech complete Menger space. Then from \cite[Theorem 1.2]{Tall2017}, $X$ is a $\sigma$-compact space. Since $X$ and $Y$ are $\mathbb{T}$-equivalent and from Theorem \eqref{t:Preservation of T equivalence}, $\sigma$-compactness preserves $\mathbb{T}$-equivalence, so $Y$ is also a Menger space.
\end{proof}

\begin{definition}
A topological group $G$ is precompact if it can be embedded as a subgroup of a compact group.
\end{definition}

\begin{theorem}[\cite{Shakhmatov2010}, Corollary 10.5]\label{t: T equivalence implies G equivalence}
For a precompact Abelian group G, $\mathbb{T}$-equivalence implies $G$-equivalence.
\end{theorem}

By using above theorem and lemma, we get following result.
\begin{theorem}
Let $G$ be a precompact Abelian group. Suppose that $X$ and $Y$ are $G$-equivalent. If $X$ is a Cech complete Menger space, then $Y$ is also a Menger space.
\end{theorem}

\begin{proof}
Suppose $X$ and $Y$ are $G$-equivalent. Take $G =\mathbb{T}$, then by Lemma \eqref{l:Menger property on T equivalence}, $\mathbb{T}$ equivalence preserves Menger property. Therefore, $Y$ is a Menger space. From Theorem \eqref{t: T equivalence implies G equivalence}, in case of a precompact Abelian group $\mathbb{T}$-eqivalence implies $G$-equivalence. Hence, $Y$ is also a Menger space.
\end{proof}

\section{Monolithicity on $\cpxg$}\label{s:$\cpxg$ is a monolithic space for a compact space $X$}

\begin{theorem}\label{s: Network weight between X and cpxg}
Let $X$ be a topological space and $G$ be a topological group which satisfies second axiom of countability. Then $nw(X)=nw(\cpxg)$
\end{theorem}

\begin{proof}
To prove this, first we will show that $nw(\cpxg)\leq nw(X)$. Fix a network $\mathcal{N}$ in $X$ and a countable base $\mathcal{B}$ in $G$. Choose $M_{1}, M_{2}, M_{3},\ldots, M_{\kappa}\in \mathcal{N} $ and $U_{1}, U_{2}, U_{3}, \ldots, U_{\kappa}\in \mathcal{B}$. We  define $W(M_{1}, M_{2}, M_{3}, \ldots, M_{\kappa}, U_{1}, U_{2}, U_{3}, \ldots,  U_{\kappa})=\{ f\in \cpxg\colon f(M_{i})\subset U_{i}, i =1, 2, 3, \dots, \kappa\}$. 

Let $N'= \{W(M_{1}, M_{2}, M_{3}, \ldots, M_{\kappa}, U_{1}, U_{2}, U_{3}, \ldots, U_{\kappa})\}$. Then we will show that $N'$ is a network of $\cpxg$. Since $\lvert N' \rvert \leq \lvert N \rvert$, it follows that $nw(\cpxg)\leq nw(X)$.

Let $f\in\cpxg$ and $V$ be the open neighbourhood of $f$ in $\cpxg$ i.e. $V=\{g\in \cpxg \colon g(x)\in U$ for some open set $U\in G\}$. Then, there exists open sets $U_{1}, U_{2}, U_{3}, \ldots, U_{\kappa}\in \mathcal{B}$ such that $U = \displaystyle\cup_{i=1}^{\kappa}U_{i}$. Since $f$ is continuous, there exists $M_{i}\in \mathcal{N}$ with $f(M_{i})\subset U_{i}$, for $i=1, 2, 3, \ldots, \kappa$ i.e. $f\in N'$. Now, we claim that $N'\subset V$. To prove this, let $g\in N'$, then $g\in \{W(M_{1}, M_{2}, M_{3}, \ldots, M_{\kappa}, U_{1}, U_{2}, U_{3}, \ldots, U_{\kappa}\}$,  i.e. $g\in \{ f\in \cpxg 
\colon f(M_{i})\subset U_{i}, i=1 ,2, 3, \ldots, \kappa\}$. Let $U=\displaystyle\cup_{i=1}^{\kappa}U_{i}$. Clearly $U$ is an open set and $g(x)\in U$, for each $x\in X$, i.e. $g\in V$. Thus $N'$ is a network of $\cpxg$. So $nw(\cpxg)\leq nw(X)$. To get the reverse inequality we use the fact that $X\subset C_{p}(\cpxg)$. So $nw(X)\leq nw( C_{p}(\cpxg))\leq nw(\cpxg)$.
\end{proof}

In the above theorem the second countability of the topological group $G$ is a sufficient condition. To see this consider the topological group $G =\mathbb{R}$ under addition with discrete topology. Any group with discrete topology is a topological group. An uncountable topological space with discrete topology cannot be a second countable space. So $G$ is a topological group which does not satisfy second countablity.

If we let $X = \mathbb{Z}$ be the set of integers with the topology induced from the usual topology of $\mathbb{R}$ and $nw(X) = \aleph_{0}$, then the space $\cpxg$ consists only of constant functions, and its cardinality is equal to that of $\mathbb{R}$, where $\mathbb{R}$ is considered with the discrete topology. Therefore, it does not have a countable network. Hence, $nw(\cpxg) \neq \aleph_{0}$

\begin{theorem}\label{t: continuity of dual map}
Let $f \colon X \to Y$ be a map and $G$ be a topological group. A map  $f^{\ast} \colon G^{Y} \to G^{X}$ defined $f^{\ast}(\phi)(x)=\phi(f(x))$ for $\phi\in G^{Y}$, then  $f^{\ast}$ is continuous.
\end{theorem}

\begin{proof}
To prove $f^{\ast}$ is continuous, let $V$ be an open neighbourhood of $\psi \in G^{X}$ and suppose that $\psi=f^{\ast}(\phi)$ for some $\phi \in G^{Y}$. Then, $V=\{ g\in G^{X} \colon g(x)\in U$ for some open set $U\in G$ for each $x\in X\}$. Since $\psi \in V$, we have $\psi(x)\in U$ for some open set $U$ in $G$. Define $V'=\{h\in G^{Y}: h(x)\in U\}$. Clearly $V'$ is an open set in $G^{Y}$.

We will prove that $f^{\ast}(V')\subset V$. Let $\theta\in f^{\ast}(V')$, then $\theta=f^{\ast}(h)$ for some $h\in V'$, i.e. $\theta(x)=h(f(x))$, and $h(y)\in U$ for some $y\in Y$. That is, $\theta(x)\in U$. So $\theta \in V$. Therefore, $f^{\ast}$ is continuous.
\end{proof}

\begin{theorem}\label{u: Homeomorphism between Gy to Gx}
Let $f:X\rightarrow Y$ be a map and $G$ be a topological group. If $f(X)=Y$, then $f^{\ast}:G^{Y}\rightarrow G^{X}$ is a homeomorphism from $G^{Y}$ onto the closed subspace $f^{\ast}(G^{Y})$ of $G^{X}$.
\end{theorem}

\begin{proof}
Let $f(X)=Y$. To prove $f^{\ast}$ is a homeomorphism, first we will show that $f^{\ast}$ is a bijective function. To see this, let $\theta_{1}, \theta_{2}\in G^{Y}$ and assume that $\theta_{1}\neq \theta_{2}$, then there exists $y\in Y$ such that $\theta_{1}(y)\neq\theta_{2}(y)$. Since $f(X)=Y$, corresponding to each $y\in Y$ there exists $x\in X$ such that $f(x)=y$. So, $f^{\ast}(\theta_{1})(x)=\theta_{1}(f(x))=\theta_{1}(y)\neq \theta_{2}(y)=\theta_{2}(f(x))= f^{\ast}(\theta_{2})(x)$, i.e. $f^{\ast}$ is one-one. Therefore, $f^{\ast}$ is a bijective function from $G^{Y}$ to $f^{\ast}(G^{Y})$.
In a similar way of the proof of Theorem $\eqref{t: continuity of dual map}$ we can easily show that $(f^{\ast})^{-1}$ is continuous. Therefore, $f^{\ast}$ is a homeomorphism from $G^{Y}$ to $f^{\ast}(G^{Y})$.

Next we will show that $f^{\ast}(G^{Y})$ is a closed set. To prove this, it is enough to show that $(f^{\ast}(G^{Y}))^{\complement}$ is an empty set. Suppose $\theta \in (f^{\ast}(G^{Y}))^{\complement}$. Then there does not exist $\psi \in G^{Y}$ such that $f^{\ast}(\psi)=\theta$.
That is, there does not exist $x\in X$ such that $\psi(f(x))=\theta(x)$. But it is a contradiction to our assumption that $f(X)=Y$. Hence $f^{\ast}(G^{Y})$ is a closed subspace of $G^{X}$.
\end{proof}

Note that in case of $G=\mathbb{R}$, $G^{\ast}$-regular space can be replaced by completely regular.

\begin{definition}[$G$-quotient map]
Let $f\colon X \to Y$ be a  map from a topological space $X$ onto a set $Y$ and $G$ be any topological group. Then the strongest of all $G^{\ast}$-regular topologies on $Y$ relative to which $f$ is continuous is called the $G$-quotient topology on the set $Y$. A map from a space $X$ onto a space $Y$ is called a $G$-quotient map if the topology on $Y$ coincide with the topology generated by $f$.
\end{definition}

\begin{example}
In case of $\mathbb{R}$ with usual topology the $G$-quotient map will coincide with the $\mathbb{R}$-quotient map and topology generated by $G$-quotient map is Real-quotient topology.
\end{example}

\begin{theorem}
Let $X$ be a compact space and $G$ be a topological group which satisfies second axiom of countability. Then the space $\cpxg$ is monolithic.
\end{theorem}

\begin{proof}
To prove $\cpxg$ is monolithic, let $M\subset \cpxg$ and $\lvert M \rvert\leq \tau$. Let $f$ be the diagonal product map of maps from $M$. Thus, $f(x) = \{ x_{g}=g(x): g\in M\}$. Let $Y=f(X)$, then $Y$ is a subspace of $G^{M}$. Therefore, $w(Y) \leq \lvert M \rvert \leq \tau$.

Let $Y'$ be the points in $Y$ with the $G$-quotient topology generated by the map $f$. Then define an identity map $i$ from $Y'$ to $Y$, making it a condensation map. So, $iw(Y')\leq w(Y)\leq \tau$. Since $X$ is a compact space, it is stable, hence it is $\tau$-stable for every infinite cardinal $\tau$. Thus, there exists a continuous function from $X$ to $Y'$, implying $nw(Y')\leq \tau$. Also, from Theorem $\eqref{s: Network weight between X and cpxg}$, we have $nw(C_{p}(Y',G))=nw(Y')\leq \tau$. Consider the map $f\colon X \to Y'$, which is evidently a $G$-quotient map. Then define $f^{\ast} \colon C_{p}(Y',G) \to \cpxg$ such that $f^{\ast}= i^{-1}\circ f$. Therefore, by Theorem \eqref{u: Homeomorphism between Gy to Gx}, $C_{p}(Y',G)$ is homeomorphic to the closed subspace $F= f^{\ast}(C_{p}(Y',G))$ of $\cpxg$.

We have $f^{\ast}= i^{-1}\circ f$, so $f = i \circ f^{\ast}$. Then for every $g\in M$, we can write $g=p_{g}\circ i \circ f^{\ast}$, where $p_{g}$ is a projection mapping from $G^{M}$ to $G$. Since projection mapping on a topological group is continuous \cite{THusain1964}, $p_{g}\circ i$ is continuous from $Y'$ to $G$, thus $p_{g}\circ i\in C_{p}(Y',G)$. This implies $g\in F$. Hence, $M\subset F$. Since $F$ is closed, $\overline{M}\subset \overline{F}=F$. $nw(\overline{M})\leq nw(F)=nw(C_{p}(Y',G))$ as $C_{p}(Y',G)$ is homeomorphic to $F$. Therefore, $nw(\overline{M})\leq \tau$. Consequently, $\cpxg$ is a monolithic space.

\end{proof}

\bibliographystyle{amsplain}
\bibliography{PropertiesOfSpaceOfGroupValuedContinuousFunctionsBib}
\end{document}